\documentclass[a4paper,11pt,DIV=12,headings=standardclasses, headings=small]{scrartcl}
\usepackage{nag}
\usepackage[utf8]{inputenc}
\usepackage[english]{babel}

\usepackage{amscd} 
\usepackage{amsmath}
\usepackage{amsthm}
\usepackage{amsfonts}
\usepackage{amssymb}
\usepackage{mathrsfs} 

\usepackage{graphicx}
\usepackage{tikz}
\usepackage{tikz-cd}
\usetikzlibrary{arrows,arrows.meta, shapes, shapes.geometric, shapes.misc, trees, graphs, matrix,  automata, backgrounds, calendar, positioning,external}
\tikzcdset{arrow style=tikz, diagrams={>=stealth}}

\usepackage{pxfonts}
\usepackage[T1]{fontenc} 
\usepackage[pdfpagelabels,hidelinks]{hyperref}

\usepackage[textsize=tiny,background color=white]{todonotes}

\usepackage[numbers]{natbib}

\DeclareMathOperator{\SO}{SO}

\DeclareMathOperator{\Ext}{Ext}

\DeclareMathOperator{\Id}{Id}
\newcommand{\id}{\text{id}}

\DeclareMathOperator{\Hom}{Hom}

\DeclareMathOperator{\Aut}{Aut}

\DeclareMathOperator{\Tr}{Tr}
\DeclareMathOperator{\fw}{fw}

\newcommand{\C}{\mathbb{C}}
\newcommand{\Z}{\mathbb{Z}}
\newcommand{\Q}{\mathbb{Q}}

\newcommand{\N}{\mathbb{N}}

\newcommand{\cC}{{\mathcal C}}

\newcommand{\De}{\Delta}

\newcommand{\va}{\varphi}

\newcommand{\ra}{{\rightarrow}}

\newcommand{\lmt}{\longmapsto}
\newcommand{\hra}{\hookrightarrow}
\newcommand{\lra}{\longrightarrow}

\newcommand{\half}{{{\frac{1}{2}}}}

\DeclareMathOperator{\B}{R}
\DeclareMathOperator{\E}{S}

\DeclareMathOperator{\haut}{hAut}
\DeclareMathOperator{\Bhaut}{BhAut}

\DeclareMathOperator{\Der}{Der}

\newcommand{\dash}{\text{-}}

\numberwithin{equation}{section}
 
\makeatletter
\newtheorem*{rep@theorem}{\rep@title}
\newcommand{\newreptheorem}[2]{%
\newenvironment{rep#1}[1]{%
 \def\rep@title{#2 \ref{##1}}%
 \begin{rep@theorem}}%
 {\end{rep@theorem}}}
\makeatother

\newtheorem{thm}{Theorem}[section]
\newreptheorem{theorem}{Theorem}

\newtheorem{cor}[thm]{Corollary}
\newtheorem{lem}[thm]{Lemma}
\newtheorem{prop}[thm]{Proposition}

\newtheorem{thmL}{Theorem}

\theoremstyle{definition}
\newtheorem{defn}[thm]{Definition}
\newtheorem{ex}[thm]{Example}
\newtheorem{rem}[thm]{Remark}

\theoremstyle{remark}

\makeatletter
\renewenvironment{proof}[1][\proofname] {\par\pushQED{\qed}\normalfont\topsep6\p@\@plus6\p@\relax\trivlist\item[\hskip\labelsep\bfseries#1\@addpunct{.}]\ignorespaces}{\popQED\endtrivlist\@endpefalse}
\makeatother

\begin{document}
\title{\Large Tautological Rings of Fibrations}
\author{\large Nils Prigge}
\date{}
\maketitle
\begin{abstract}\noindent
\textbf{Abstract.} We study the analogue of tautological rings of fibre bundles in the context of fibrations with Poincar\' e fibre, i.e.\ the ring obtained by fibre integrating powers of the fibrewise Euler class. We discuss how to compute the Euler ring with tools from rational homotopy theory and completely determine the tautological ring for even spheres, complex projective spaces and some products of odd spheres.
\end{abstract}

\tableofcontents

\section{Introduction}
Let $\pi: E\ra B$ be a fibration with fibre $X$ an oriented Poincar\' e duality space of formal dimension $d$ (cf.\,\cite[Ch.\,1]{Wa67})\footnote{If $X$ is a simply connected finite CW complex, this just means that we (can) choose a fundamental class $[X]\in H_d(X;\Z)$ that induces the Poincar\'e duality isomorphism for all local coefficient systems.}. It is called \emph{oriented} if the corresponding local coefficient system $\mathcal{H}^d(X)$ is trivial and we choose an isomorphism $\mathcal{H}^d(X)\cong \Z$. In \cite[Sect.\,3]{HLLRW17} the authors construct for such fibrations a \emph{fibrewise Euler class} $e^{\fw}(\pi)\in H^d(E)$ which extends the construction for smooth fibre bundles, i.e.\,if $X$ is a closed, oriented manifold and $\pi:E\ra B$ is an oriented fibre bundle, then the fibrewise Euler class agrees with the Euler class of the vertical tangent bundle $T_{\pi}E\ra E$ which, if $\pi$ is a smooth submersion, is defined as the kernel of the differential of the projection map $T_{\pi}E:=\ker(D\pi:TE\ra TB)\subset TE$.

\smallskip
Using the fibrewise Euler class we can extend the construction of \emph{tautological classes} of smooth fibre bundles \cite{Mum83,RW16,GGRW17,Gri17} to fibrations. Recall that given a smooth, oriented fibre bundle $\pi:E\ra B$ with fibre $M$ a closed manifold and a class $c\in H^{|c|}(\mathrm{B}\SO(d))$, the fibre integral $\kappa_c(\pi):=\int_{\pi}c(T_{\pi}E)\in H^{|c|-d}(B)$ is a characteristic classes of the bundle and the subring of $H^*(B)$ generated by all such tautological classes is called the \emph{tautological ring} $R^*(\pi)$. 

As we can define fibre integration more generally for oriented fibrations with Poincar\'e fibre (see \cite[Sect.\,8]{BH58} and Section \ref{FibreIntegration}), we obtain analogous characteristic classes of oriented fibrations $\pi:E\ra B$ with Poincar\'e fibre by setting
\[\kappa_i(\pi):=\pi_!(e^{\fw}(\pi)^{i+1})\in H^{i\cdot d}(B),\]
where $\pi_!:H^*(E)\ra H^{*-d}(B)$ is another notation for the fibre integration map. In particular, if 
\begin{equation}\label{UnivXFibr}
X\hra E\overset{\pi}{\lra} \Bhaut^+(X),
\end{equation}
denotes the universal oriented $X$-fibration that classifies oriented $X$-fibrations \cite{St63} over the classifying space $\Bhaut^+(X)$ of the monoid of orientation preserving homotopy self-equivalences $\haut^+(X)$, we can study the subring generated by the tautological classes.
\begin{defn}
  Let $X$ be an oriented Poincar\' e duality space of formal dimension $d$. The \emph{Euler ring} $E^*(X)$ is the subring of $H^*(\Bhaut^+(X))$ generated by all tautological classes
  \begin{equation}
  \kappa_i:=\pi_!(e^{\text{fw}}(\pi)^{i+1})\in H^{id}(\Bhaut^+(X)),
  \end{equation}
  where $e^{\text{fw}}(\pi)\in H^d(E)$ is the fibrewise Euler class of the universal oriented $X$-fibration in \eqref{UnivXFibr}. 
\end{defn}
Unlike the smooth tautological ring, the computation of the Euler ring $E^*(X)$ is a purely homotopy theoretic problem. The main content of this paper is to work out how one can use the models from rational homotopy theory for fibrations to compute Euler rings. Specifically, we determine representatives of the fibre integration maps and fibrewise Euler classes in terms of the algebraic models from rational homotopy theory. The computation of the Euler ring turns out to be particularly tractable for rationally elliptic spaces and we determine the Euler ring for two classes of examples.
\begin{thmL}
 The Euler ring of complex projective space is $E^*(\C P^n)\cong \Q[\kappa_1,\hdots,\kappa_{n-1},\kappa_{n+1}]$.
\end{thmL}
\begin{thmL}
 Let $X$ be either rationally equivalent to $S^{2k+1}\times \hdots \times S^{2k+1}$ for $k\geq 1$ or a finite CW complex rationally equivalent to a product of two odd dimensional simply connected spheres of different dimension. Then $E^*(X)=\Q$.
\end{thmL}

Some further computations of Euler rings can be found in the author's thesis \cite[Sect.\,4.2]{Pri20}. One can also extend the definition of the Euler ring for fibrations with extra structure to obtain, for example, better homotopy theoretic approximations to the smooth tautological ring \cite{Ber20}. Finally, these techniques can be used to infer properties about smooth tautological rings \cite{Pri23}.

\paragraph{Acknowledgements.} This paper is part of my PhD project and I would like to thank my supervisor Oscar Randal-Williams for many enlightening discussions and his help and patience. I would also like to thank Alexander Berglund for many useful discussions and for his help with the revision of this article, and Andrey Lazarev for pointing out modern references for models of the universal fibration. For the revision of this article I have been supported by the Knut and Alice Wallenberg foundation through grant no.\ 2019.0519.

\section{Rational homotopy theory of fibrations}
The classifying space $\Bhaut^+(X)$ is rarely simply connected even if $X$ is, and therefore we cannot immediately apply the results from rational homotopy theory. Instead, for a simply connected Poincar\'e duality space $X$ we study the universal 1-connected fibration 
\begin{equation}\label{UnivXFibr1connected}
 X\hra E\overset{\pi}{\lra}\Bhaut_0(X),
\end{equation}
over the classifying space of the path component of the identity $\haut_0(X)\subset \haut^+(X)$ (or equivalently the induced fibration over the universal covering of $\Bhaut^+(X)$). Many people have studied rational models of \eqref{UnivXFibr1connected} (see \cite{Sul77,Ta83,La14}) and fibrations in general (see \cite{Ha83,FHT}), and we will use these algebraic models to compute the image of the Euler ring
\begin{equation}
E^*_0(X)\subset H^*(\Bhaut_0(X))
\end{equation}
induced by the natural map $\Bhaut_0(X)\ra \Bhaut^+(X)$. We discuss in Section \ref{computations} how one can in some cases upgrade the computation of $E_0(X)$ to a computation of the full Euler ring $E^*(X)$.

\subsection{Rational models of fibrations}\label{ClassicalModels}
From now on we consider cohomology with rational coefficients unless stated otherwise and we use standard terminology from rational homotopy theory as discussed in \cite{FHT}. Throughout this paper, we denote a cdga model of a map of connected spaces $\pi:E\ra B$, typically a fibration with simply connected fibre, by $\pi^*:\B\ra \E$ and we also assume, unless stated otherwise, that cgda models are connected (i.e.\,$R^0=S^0=\Q$). The following cdga models for fibrations enjoy good homotopical properties.

\begin{defn}\cite[Sect.\,14]{FHT}
A \emph{relative Sullivan algebra} is a cdga $(\B\otimes \Lambda V,D)$ so that $\id_{\B}\otimes 1:\B\ra \B\otimes \Lambda V$ is a map of cdga's, and $V=\oplus_{p\geq 1}V^p$ is a graded vector space with an exhaustive filtration $V(0)\subset V(1)\subset\hdots$ of graded subspaces so that $D|_{1\otimes V(0)}:V(0)\ra \B$ and $D|_{1\otimes V(k)}:V(k)\ra \B\otimes \Lambda V(k-1)$. A \emph{relative Sullivan model} of a map of cdgas $\pi^*:\B\ra \E$ is a relative Sullivan algebra  $\E=(\B\otimes \Lambda V,D)$ together with a quasi-isomorphism $\E'\overset{\simeq}{\ra} \E$ of $\B$-algebras.
\end{defn}
By \cite[Prop.\,14.3]{FHT} any map $\pi^*:\B\ra \E$ of connected cdga's admits a relative Sullivan model if $H(\pi^*):H^1(\B)\ra H^1(\E)$ is injective. Since this condition is always satisfied if $\pi^*$ is the model of a fibration $\pi:E\ra B$ with connected fibre, we can always find relative Sullivan models in this case. Moreover, if either the base or fibre has finite type and $\pi_1(B)$ acts trivially on the cohomology of the fibre, then $\E\otimes_{\B}\Q$ is a Sullivan model of the fibre by \cite[Thm 20.3]{Ha83}.

\smallskip
There is a convenient way to obtain relative Sullivan algebras via differential graded Lie algebras. We follow \cite{Ber20} in our convention (which in turn is based on \cite{Ta83}), in defining  Chevalley-Eilenberg complex of a dg Lie algebra $L$ as the dg coalgebra \[\cC^{CE}(L)=(\Lambda^c sL,D=d_1+d_2),\] where $\Lambda^c sL$ denotes the cofree, conilpotent, cocommutative coalgebra on the suspension $(sL)_*=L_{*-1}$ and the differentials are determined by their corestrictions
\begin{align*}
 d_1(sl)&=-sd_L(l), & d_2(sl_1\wedge sl_2)=(-1)^{|l_1|}s[l_1,l_2].
\end{align*}
If a dg Lie algebra acts on a cdga $A$ through derivations, then the Chevalley-Eilenberg cochain complex is
\[\cC_{CE}(L;A):=\big(\Hom(\cC^{CE}(L),A),\partial+t\big),\]
where 
\begin{align}\label{CEdifferential}
\begin{split}
 \partial(f)&=d_A\circ f-(-1)^{|f|}f\circ D\\
 t(f)(sl_1\wedge \hdots \wedge sl_n)&=\sum_{i=1}^n(-1)^{|sl_i|(|f|+|sl_1|+\hdots+|sl_{i-1}|)}l_i\cdot f(sl_1\wedge \hdots \hat{sl_i} \hdots \wedge sl_n),
\end{split}
\end{align}
and $\cC_{CE}(L;A)$ is a cdga via the convolution product. An element $f\in \cC^{CE}(L;A)$ is an \emph{$n$-cochain} if $f(sl_1\wedge\hdots\wedge sl_k)=0$ unless $k=n$ and we identify   $0$-cochains with $A$. Moreover, if $A$ is a Sullivan algebra then $\cC_{CE}(L,A)$ is a relative Sullivan algebra over the Chevalley-Eilenberg cochain complex $\cC_{CE}(L):=\cC_{CE}(L;\Q)$ via the map on Chevalley-Eilenberg cochain complexes induced by the unit $\eta:\Q\ra A$.

\bigskip
With these definitions in place, we can describe a cdga model of \eqref{UnivXFibr1connected}. Let $X$ be a simply connected space of finite type and with minimal Sullivan model $(\Lambda V,d)$. Consider the dg Lie algebra $(\Der^+(\Lambda V),[d,-])$ of positive degree derivations, where a derivation $\theta\in \Der(\Lambda V)$ has degree $n$ if it \emph{lowers}\footnote{This slightly confusing convention is due to using homological grading conventions for dg Lie algebras and cohomological grading conventions for cdga's and can be avoided if one sticks to just one.} the degree by $n$ (additionally, $\theta\in \Der^+(\Lambda V)_1$ only if $[d,\theta]=0$). It was first proved by Sullivan \cite{Sul77} that $\Der^+(\Lambda V)$ is a dg Lie model of $\Bhaut_0(X)$, and one can further describe a model of \eqref{UnivXFibr1connected} as follows.

\begin{thm}\label{modelcdga}
 Let $X$ be a 1-connected space of finite type with minimal Sullivan model $(\Lambda V,d)$ and unit $\eta:\Q\ra \Lambda V$. Then
\begin{equation}\label{modelcdgaeq}
   \cC^*_{CE}(\Der^+(\Lambda V);\Q)\overset{\eta_*}{\lra}\cC^*_{CE}(\Der^+(\Lambda V);\Lambda V)
  \end{equation}
is a relative Sullivan model of the universal oriented 1-connected fibration \eqref{UnivXFibr1connected}. 
 \end{thm}

\begin{rem}\leavevmode
\begin{itemize}
 \item[(i)] In a previous version of this article we proved this result via a comparison of \eqref{modelcdgaeq} with Tanr\'e's model \cite{Ta83} of the universal 1-connected fibration (see also \cite{Pri20}). It has been pointed out to the author by Andrey Lazarev that instead one can derive this result more directly from \cite{La14}. Another proof is due to Alexander Berglund and is based on rational models for the bar construction in terms of Chevalley-Eilenberg complexes \cite{Ber17}. The statement of Theorem \ref{modelcdga} can be found as a special case of \cite[Prop.\,3.6]{Ber20} and we refer to these papers for a proof.
 \item[(ii)] Theorem \ref{modelcdga} turns out to be particularly useful for rationally elliptic spaces because the dg Lie algebra $\Der^+(\Lambda V)$ is finite dimensional in contrast to Tanr\'e's model which is only of finite type. However, if $X$ is rationally hyperbolic then $\Der^+(\Lambda V)$ is not even of finite type and it seems more feasible to study other models of \eqref{UnivXFibr1connected} instead, for example based on a dg Lie algera model of the fibre (see for example \cite{BM14,Sto22}).
\end{itemize}
\end{rem}

\begin{ex}\label{EvenSphereFibrationModel}
Let $X=S^{2n}$ with Sullivan model 
\[A_n=(\Lambda (x,y),|x|=2n,|y|=4n-1,d=x^2\cdot \partial\slash\partial y).\]
Then $\Der^+(A_n)$ is $3$-dimensional with basis $\eta_{2n-1}:=x\cdot \partial\slash\partial y$, $\eta_{2n}:=\partial\slash\partial x$ and $\eta_{4n-1}:={\partial\slash \partial y}$ and differential $[d,\eta_{2n}]=-2\eta_{2n-1}$. Hence, the inclusion of the abelian Lie algebra with trivial differential $\mathfrak{g}:=\Q\{\eta_{4n-1}\}\hra \Der^+(A_n)$ is a quasi-isomorphism of dg Lie algebras and we get a cdga quasi-isomorphism $\cC^*_{CE}(\Der^+(A_n);\Q)\ra \cC^*_{CE}(\mathfrak{g};\Q)=(\Lambda z_{4n},d=0)$ as well as with coefficients in $A_n$. Thus, the cgda model of the universal 1-connected $S^{2n}$-fibration in Theorem \ref{modelcdga} is equivalent to
\begin{align*}
 (\Lambda (z_{4n}),d=0)\lra (\Lambda (z_{4n})\otimes \Lambda (x,y),D(x)=0,D(y)=x^2+z_{4n}),
\end{align*}
where $z_{4n}$ is the 1-cochain dual to $s\eta_{4n-1}$.
\end{ex}

\section{Fibre integration in rational homotopy theory}\label{FibreIntegration}
Before we discuss rational models for fibre integration, we recall the definition suitable for oriented fibrations as a special case of the following construction: Let $\pi:E\ra B$ be a fibration with fibre $X$ and $H^*(X)=0$ for $*>d$ and $H^d(X)$ nontrivial. Given a $\pi_1(B)$-module homomorphism $\phi:\mathcal{H}^d(X)\ra \Q$, we can define \emph{$\phi$-integration} as the composition
\begin{equation}\label{FibreFunction}
\phi_!:H^*(E)\twoheadrightarrow E_{\infty}^{*-d,d}\subset E_2^{*-d,d}=H^{*-d}(B;\mathcal{H}^d(X))\overset{H(\phi)}{\lra} H^{*-d}(B),
\end{equation}
where we project $H^*(E)$ onto the $d$-th row of the $E_{\infty}$-page of the Serre spectral sequence, which is possible since $H^*(X)=0$ for $*>d$, and $E_{\infty}^{*-d,d}\subset E_2^{*-d,d}$ as there are no differentials into this row. 

Note that because the cohomological Serre spectral sequence is compatible with cup product, there is a push-pull identity
\begin{equation}\label{pushpull}
\phi_!(\pi^*(x)\smile y)=x\smile \phi_!(y)
\end{equation}
for $x\in H^*(B)$ and $y\in H^*(E)$, and thus $\phi_!$ is a $H^*(B)$-module map. 

\begin{defn}\label{FibreIntegrationDef}
	Let $\pi:E\ra B$ be an orientable fibration with Poincar\'e fibre $X$ of formal dimension $d$ and let $\varepsilon_X:H^d(X)\ra \Q$ be an orientation of $X$. Then 
	\begin{equation*}
	\pi_!:=(\varepsilon_X)_!:H^*(E)\ra H^{*-d}(B)
	\end{equation*}
	is called \emph{fibre integration} of $\pi:E\ra B$.
\end{defn}

\subsection{Chain level fibre integration}

In the main result of this section we show that there exists a chain level representative of fibre integration. More precisely, given a relative Sullivan model $\E=(\B\otimes\Lambda V,D)$ of an oriented fibration $\pi:E\ra B$ with Poincar\'e fibre $X$, there exists a $\B$-module homomorphism $\Pi:\E\ra\B$ of degree $-d$ which is unique up to homotopy and induces fibre integration on cohomology. This is most conveniently expressed in terms of differential $\Ext$ groups.

\medskip
Let $\B$ be a connected cdga and $M,N$ a differential graded $\B$-modules. Then $\B$-module homomorphisms $\Hom_{\B}(M,N)$ is a $\B$-module with differential $D(f):=d_Nf-(-1)^{|f|}fd_M$. Recall from \cite[Sect.\,7]{FHT} that $M$ is a \emph{semifree $\B$-module} if there is an exhaustive filtration $M(0)\subset M(1)\subset\hdots M$ so that $M(0)$ and $M(k)/M(k-1)$ are free $\B$-modules for all $k\geq 1$. A \emph{semifree resolution} of a $\B$-module $M$ is a semifree $\B$-module $M'$ with a quasi-isomorphism $M'\xrightarrow{\simeq} M$ and we denote by $M\otimes_{\B}^{\mathbb{L}}N$  the derived tensor product of two $R$-modules given by $M'\otimes_RN$ for some semifree resolution $M$. Following \cite{FMT}, the differential Ext groups for $\B$-modules $M,N$ are defined as 
\begin{equation}\label{ExtDef}
 \Ext_{\B}(M,N):=H(\Hom_{\B}(M',N))
\end{equation}
for a semifree resolution $M'\xrightarrow{\simeq} M$.

\begin{prop}\label{ExtModuleMaps}
	Let $\pi:E\ra B$ be a fibration with connected base and total space and 1-connected fibre $X$. Assume that $\pi_1(B)$ acts trivially on $H^*(X)$ and that $H^*(X)$ is of finite type and nontrivial in degree $d$ and vanishes for $*>d$. Let $\pi^*:\B\ra \E$ be a cdga model of $\pi$, then the augmentation induces an isomorphism 
	\begin{equation}\label{PhiIntegration}
	 \Ext_{\B}^{-d}(\E,\B)\overset{\cong}{\lra} \Ext_{\Q}^{-d}(\E\otimes_{\B}^{\mathbb{L}}\Q,\Q)\cong \Hom(H^d(X),\Q),
	\end{equation}
	and given $\phi\in \Hom(H^d(X),\Q)$ the chain level representative $[\Phi]\in\Ext_{\B}^{-d}(\E,\B)$ induces $\phi$-integration.
\end{prop}

\begin{rem}
 Proposition \ref{ExtModuleMaps} generalizes \cite[Thm.A]{FT09} where they identify fibre integration as elements in differential Ext groups for fibrations over Poincar\'e duality spaces and for pullbacks from such fibrations. Moreover, fibre integration can be identified rationally as a map of parametrized suspension spectra $\pi_!:\Sigma_B^{\infty}B_+\ra \Sigma_B^{\infty-d}E_+$, and by \cite[Thm.1.1]{FMT} the set of homotopy classes of such maps is given by differential Ext groups which is consistent with our result.
\end{rem}

\begin{proof}
    Let  $\E'\xrightarrow{\simeq} \E$ be a relative Sullivan model. Then $\Ext^{*}(\E,\B)=H^*(\Hom_{\B}(\E',\B))$ since $\E'$ is a semifree resolution by \cite[Lem.\,14.1]{FHT}. We consider the exhaustive filtration of $\Hom_{\B}(\E',\B)$ given by $F^p=\Hom_{\B}(\E',\B^{\geq p})$. According to \cite[Thm.9.3]{Bo98} the corresponding spectral sequence is conditionally convergent to the completion
	\begin{align*}
	\underset{\longleftarrow}{\lim}_p \Hom_{\B}(\E',\B)\slash \Hom_{\B}(\E',\B^{\geq p}) &\cong  \underset{\longleftarrow}{\lim}_p \Hom_{\B}(\E',\B\slash \B^{\geq p})\\
	&\cong \Hom_{\B}(\E', \underset{\longleftarrow}{\lim}_p\B\slash\B^{\geq p})\\
	&\cong \Hom_{\B}(\E',\B)
	\end{align*}
	where we have used that $\E'$ is a projective $\B$-module for the first isomorphism. The $E_1$-page is $H(\Hom_{\B}(\E',\Q))\otimes B^p$. By the assumption on the fibration it follows from \cite{Ha83} that the Sullivan fibre $\E'\otimes_{\B}\Q=(\Lambda V,d)$ is a cdga model of $X$. Hence, the $E_1$-page can be simplified as $E_1^{p,q}=H^q((\Lambda V)^{\vee})\otimes \B^p\cong \Hom^q(H^*(X),\Q)\otimes \B^p$. Since $\pi_1(B)$ acts trivially, the differential on the $E_1$-page is given by $\id\otimes d_{\B}:E_1^{p,q}\ra E_1^{p+1,q}$ (if $\B$ was simply connected this follows for degree reasons but in general it follows from the fact the cdga model of the universal 1-connected $X$-fibration in Theorem \ref{modelcdga} has a simply connected base and by assumption the relative Sullivan model is obtained by base change along a map $\cC_{CE}(\Der^+(\Lambda V))\ra \B$). Hence the $E_2$-page is $E_2^{p,q}= H^p(\B)\otimes H^q((\Lambda V)^{\vee})$.
	In particular, the spectral sequence vanishes for $q<-d$. Since the gradings are such that the differentials are $d_r:E^{p,q}_r\ra E^{p+r,q-r+1}$, there are only finitely many nontrivial differentials. This implies that the derived $E_\infty$-page is zero and so by \cite[Thm 7.1]{Bo98} the spectral sequence converges strongly, 
	\begin{equation*}
	E_2^{p,q}= H^p(\B)\otimes H^q((\Lambda V)^{\vee})\cong H^p(B)\otimes \Hom^q(H^{*}(X),\Q)\Rightarrow H(\Hom_{\B}(\E',\B)),
	\end{equation*}
	and we can recover $H(\Hom_{\B}(\E',\B))$ from the entries of the $E_{\infty}$-page. The only contribution with total degree $-d$ comes from $E_{\infty}^{0,-d}\cong E_2^{0,-d}\cong \Hom(H^d(X),\Q)$ which proves the first part of the statement.
	
	It remains to show that for $\phi=H^d(\Phi\otimes_{\B} \Q):H^d(\Lambda V,d)\ra \Q$ the induced $\phi$-integration map coincides with $H(\Phi):H^*(\E')\ra H^*(\B)$. First, we note that $\E'=(\B\otimes \Lambda V,D)$ has a filtration $F^p=\B^{\geq p}\otimes \Lambda V$ and that the corresponding spectral sequence converges as $E_2^{p,q}=H^p(\B)\otimes H^q(\Lambda V) \Rightarrow H^{p+q}(\E')$. In fact, $\B$ also has an analogous filtration $G^p=\B^{\geq p}$ with only nontrivial differential on the $E_1$-page. Then $\Phi$ induces a map between these two filtrations and the map on $E_2$-pages is precisely $\phi$-integration defined using this spectral sequence. As we have defined $\phi$-integration using the Serre spectral sequence, it remains to show that this spectral sequence is isomorphic to the Serre spectral sequence. Grivel has shown in \cite{Gr79} that the above filtration gives rise to the Serre spectral squence if the base is simply connected, and this has been generalized by Halperin \cite{Ha83} if $\pi_1(B)$ acts nilpotently on the cohomology of the fibre. More precisely, it follows from the proof of \cite[Thm.6.4]{Gr79} respectively \cite[Sect.\,20]{Ha83} that the comparison map of $\B\ra \E'$ with $A_{PL}(B)\ra A_{PL}(E)$ is compatible with Dress' construction of the Serre spectral sequence \cite{Dr67} and induces an isomorphism on the $E_2$-pages, which concludes the proof.
\end{proof}
We can use Proposition \ref{ExtModuleMaps} to build a representative of fibre integration for an oriented fibration $\pi:E\ra B$ and oriented Poincar\'e fibre $(X,\varepsilon_X)$ as follows: Consider a relative Sullivan model $\pi^*:\B\ra (\B\otimes \Lambda V,D)$. Pick a chain level representative $\varepsilon$ of the orientation $\varepsilon_X\in H^{-d}(\Hom(\Lambda V,\Q))$ of $X$. By Proposition \ref{ExtModuleMaps} there is a cycle $\Pi\in \Hom^{-d}(\E',\B)$ unique up to chain homotopy that satisfies
\begin{equation}\label{normalization}
\Pi(1\otimes \chi)=\varepsilon_X(\chi) \in\B^0=\Q
\end{equation}
for all $\chi \in (\Lambda V)^d$ and that induces fibre integration on cohomology
\begin{equation*}
\pi_!:H^*(E;\Q)\cong H^*(\E')\overset{H(\Pi)}{\lra} H^{*-d}(\B)\cong H^{*-d}(B;\Q).
\end{equation*}
We demonstrate this technique in the following example.

\begin{ex}\label{EvenSphereFibreIntegration}
	Recall the relative Sullivan model of the universal 1-connected fibration for an even dimensional sphere $X=S^{2n}$ as discussed in Example \ref{EvenSphereFibrationModel}. We choose as orientation $\varepsilon_X:A_n\ra \Q$ the homomorphism determined by $\varepsilon_X(x)=1$. For degree reasons $\Pi(y x^k)=0$ and since $\Pi$ has to be a chain map we have $0=\Pi(D(yx^k))=\Pi(x^{k+2}+z_{4n}x^k)$. This determines a $\Lambda(z_{4n})$-module map $\Pi:(\Lambda(z_{4n},x,y),D)\ra (\Lambda(z_{4n}),d=0)$ by 
	\begin{equation*}
	\Pi(yx^k)=0 \qquad \text{and} \qquad \Pi(x^n)=\begin{cases} 0 & n=2k \\ (-1)^kz_{4n}^k & n=2k+1 \end{cases}
	\end{equation*}
	which is a chain map by construction and induces fibre integration on cohomology as it satisfies \eqref{normalization}.
\end{ex}

\section{The fibrewise Euler class}
The definition of the fibrewise Euler class in \cite[Def.\,3.1.1.]{HLLRW17} uses constructions in the category of parametrized spectra. And while there has been a lot of progress to adapt the tools from rational homotopy theory to the context of parametrized stable homotopy theory \cite{BM21,BM20}, there is a simpler way to define the fibrewise Euler class with rational coefficients so that we can avoid discussing the category of parametrized spectra and their rational models altogether.

\smallskip
We begin by describing a special case of the definition of the fibrewise Euler class. Consider an oriented fibration $\pi:E\ra B$ with fibre $X$ so that both $X$ and $B$ are orinted Poincar\'e duality spaces of dimensions $d$ and $b$ respectively. Then the total space is also an oriented Poincar\'e duality space \cite{Go79}, and we can define the Umkehr map of the fibrewise diagonal $\De:E\ra E\times_BE$ by
\begin{equation}
 \De_!:H^*(E)\underset{\cong}{\xrightarrow{D_E}}H_{b+d-*}(E)\xrightarrow{\De_*}H_{b+d-*}(E\times_BE)\underset{\cong}{\xrightarrow{D^{-1}_{E\times_BE}}}H^{*+d}(E\times_BE)
\end{equation}
where $D_E$ and $D_{E\times_BE}$ denote the Poincar\'e duality isomorphisms. It is shown in \cite[Sect.\,3]{HLLRW17} that the fibrewise Euler class of $\pi:E\ra B$ agrees with 
\begin{equation}\label{efwM}
 e^{\fw}(\pi)=\De^*(\De_!(1))\in H^d(E).
\end{equation}
This is sufficient to define the fibrewise Euler class with rational coefficients, because for any space $X$ the rational homology groups are isomorphic to rationalized stable framed bordism $H_*(X;\Q)\cong \Omega^{\mathrm{sfr}}_*(X)\otimes \Q$, and so we can determine a rational cohomology class by defining its evaluation on framed bordism classes. Hence, denoting by $E$ the total space of the universal 1-connected fibration \eqref{UnivXFibr1connected}, given a stably framed bordism class $[f:M^d\ra E,\xi]\in \Omega^{\mathrm{sfr}}_d(E)\otimes \Q$, we can consider the pullback of the $X$-fibration $p_1:E\times_{\Bhaut_0(X)}E\ra E$ along $f$. The pullback $\pi:f^*(E\times_{\Bhaut_0(X)}E)\ra M$ has a section $s$ via the diagonal, and we can associate to it an Euler class $e_f$ via \eqref{efwM}. Then the fibrewise Euler class $e^{\fw}(\pi)\in H^d(E;\Q)$ agrees with the class defined by the pairing $\langle e^{\fw}(\pi),[f:M\ra E,\xi]\rangle =\langle s^*(e_f),[M]\rangle$ since the construction in \cite{HLLRW17} coincides with \eqref{efwM} if the base is a Poincar\'e duality space.

\subsection{A cocycle representative via rational homotopy theory}
The idea to obtain cocycle representatives of the fibrewise Euler class is simply to construct a chain level representative of the Umkehr map $\De_!:H^*(E)\ra H^{*+d}(E\times_BE)$ in terms of the algebraic models. In the following, we denote by $\Q[n]$ the graded vector space with $\Q$ in degree $n$ and for a $\B$-module $M$ define $M[n]:=M\otimes \Q[n]$.
\begin{prop}\label{fwPD}
 Let $\B\ra \E$ be a relative Sullivan model of an oriented fibration $\pi:E \ra B$ with connected base and total space and simply connected Poincar\'e fibre $X$ of formal dimension $d$. Let $\varepsilon_X$ be the orientation and $\Pi\in \Hom_{\B}^{-d}(\E,\B) $ be a cocycle representative of fibre integration. Then the map
 \begin{align*}
 \bar{\Pi}:\E[-d]&\lra \Hom_{\B}(\E,\B)\\
  e&\lmt (e'\mapsto (-1)^{d+d\cdot|e|} \Pi(e\cdot e'))
 \end{align*}
is a quasi-isomorphism of $\B$-modules.
\end{prop}
\begin{proof}
It is a simple check that $\bar{\Pi}$ defines a $\B$-module homomorphism. By assumption, $\E=(\B\otimes \Lambda V,D)$ is a relative Sullivan algebra and thus has a filtration which induces the Serre spectral sequence as discussed in the proof of Proposition \ref{ExtModuleMaps}. In the same proof we have described a filtration of $\Hom_{\B}(\E,\B)$ which strongly converges because there is a horizontal vanishing line. The map $\bar{\Pi}$ is compatible with the two filtrations and induces a map of the associated spectral sequences. The induced map on the $E_2$-page is given by
\begin{align*}
E_2^{p,q}=H^p(\B)\otimes H^q(\Lambda V,d)\overset{\Id\otimes \bar{\varepsilon}_X}\lra H^p(\B)\otimes  \Hom(H^{q-d}(\Lambda V),\Q)
\end{align*}
where $\bar{\varepsilon}_X:H^q(\Lambda V)\ra \Hom(H^{d-q}(\Lambda V),\Q)$ is the adjoint of $H^q(\Lambda V)\otimes H^{d-q}(\Lambda V)\overset{\cup}{\ra}H^d(\Lambda V)\overset{\varepsilon_{X}}{\ra} \Q$. Since $(H^*(X;\Q),\varepsilon_X)$ is an oriented Poincar\'e duality algebra, $\bar{\Pi}$ induces an isomorphism of $E_2$-pages.
\end{proof}
This enables us to define an algebraic Umkehr map as follows. Let $\B\ra\E$ be a relative Sullivan model of $\pi:E\ra B$ and $\Pi:\E\ra\B$ a chain level representative of fibre integration. Then $\E\otimes_{\B}\E$ is a Sullivan model of $E\times_BE$, the multiplication $\mu:\E\otimes_{\B}\E\ra \E$ is a model of the fibrewise diagonal $\De:E\ra E\times_BE$ and $\Pi\otimes \Pi:\E\otimes_{\B}\E  \ra \B\otimes_{\B}\B=\B$ is a chain level representative of fibre integration for $E\times_BE$. Since $\E$ is $\B$-semifree, we can find a lift of $\B$-modles
 \begin{equation}\label{UmkehrDiagram}
 \begin{tikzcd}[ampersand replacement=\&]
 \E'[-d]\arrow{d}{\simeq} \arrow[swap]{d}{\bar{\Pi}}\arrow[dashed]{r}{\De_!}\& \E'\otimes_{\B}\E'[-2d] \arrow{d}{\simeq}\arrow[swap]{d}{\overline{\Pi\otimes\Pi}} \\
 \Hom_{\B}(\E',\B)\arrow{r}{\mu^*}  \& \Hom_{\B'}(\E'\otimes_{\B}\E',\B)
 \end{tikzcd}
 \end{equation}
 which is unique up to homotopy and therefore obtain a well defined class
 \begin{equation}\label{efwAlgebraic}
 [\mu(\De_!(1)]\in H^d(\E).
 \end{equation}

\begin{prop}\label{EulerClassDef}
 Let $\pi:E\ra B$ be an oriented fibration with connected base and total space and simply connected Poincar\'e fibre $X$ of formal dimension $d$ so that $\pi_1(B)$ acts trivally on $H^*(X)$. If $\B\ra\E$ is a relative Sullivan model and $\De_!:\E\ra\E\otimes_{\B}\E$ an Umkehr map as in \eqref{UmkehrDiagram}, then $\mu(\De_{!}(1))\in \E^d$ is a representative of the fibrewise Euler class $e^{\fw}(\pi)\in H^d(E)$.
\end{prop}

\begin{proof}
We first observe that the definition of the class in \eqref{efwAlgebraic} is natural with respect to pullbacks: For a map $f:B'\ra B$ with cdga model $\phi:\B\ra \B'$, a model of the pullback $\pi:f^*E\ra B'$ is given by $\B'\otimes_{\B}\E$ and a model on the map of total spaces is given by sending $\Phi(s)=1\otimes s\in \B'\otimes_{\B}\E$ for $s\in \E$ by \cite[Sect.\,20.6]{Ha83}. It follows from Proposition \ref{ExtModuleMaps} that a cocycle representative of fibre integration is given by $\B'\otimes_{\B}\Pi:\B'\otimes_{\B}\E\ra \B'\otimes_{\B}\B\cong \B'$ and therefore $\B'\otimes_{B}\Delta_!$ is a model of the Umkehr map so that \eqref{efwAlgebraic} in this case is given by $[1\otimes \mu(\De_!(1))]=\Phi(\mu(\De_!(1)))\in \B'\otimes_{\B}\E$.

\smallskip
Hence, it suffices to prove that \eqref{efwAlgebraic} coincides with the class defined in \eqref{efwM} for fibrations where the base space is a Poincar\'e duality space (or even just for closed, stably framed manifolds). So let us assume that the base space is a closed manifold $B$ or more generally a Poincar\'e duality space. Then $B\ra *$ is a fibration with Poincar\'e fibre and we can apply the results of the previous section to get a chain level representative of fibre integration map $\Pi_B\in \Hom_{\Q}(\B,\Q)$ corresponding to evaluating a fundamental class. Suppose $\B\ra\E$ is a Sullivan model of the fibration and let $\Pi\in \Hom_{\B}(\E,\B)$ be a model of fibre integration of $\pi$. Then $([B]\cap \dash) \circ\pi_!:H^{b+d}(E)\ra H^{b}(B)\ra \Q$ is an orientation of the Poincar\'e algebra $H^*(E;\Q)$ itself and therefore $\Pi_B\circ \Pi\in \Hom_{\Q}(\E,\Q)$ is a cocycle representative. Define $\Pi_E:=\Pi_B\circ \Pi$ so that
 \begin{equation*}
\begin{tikzcd}[ampersand replacement=\&]
   \Hom_{\B}(\E,\B) \arrow[swap,bend right=15]{rr}{(\Pi_B)_*} \& \E[-d] \arrow[swap]{l}{\simeq}\arrow{l}{\bar{\Pi}}  \arrow{r}{\simeq}\arrow[swap]{r}{\bar{\Pi}_E} \& \Hom_{\Q}(\E),\Q)
\end{tikzcd}
 \end{equation*}
is a commuting diagram, where $(\Pi_B)_*\va=\Pi_B\circ \va $ for $\va\in \Hom_{\B}(\E,\B))$. We can choose chain level representatives of fibre integration of $E\times_BE\ra B $ and $E\times_BE\ra *$ as $\Pi\otimes\Pi:\E\otimes_{\B}\E\ra \B\otimes_{\B}\B\cong \B$ and $\Pi_{E\times_BE}:=\Pi_B\circ (\Pi\otimes\Pi)$ by the same arguments as above. We therefore have a diagram 
\begin{equation*}
 \begin{tikzcd}[ampersand replacement=\&]
  \E[-d] \arrow[bend right=45,shift right=26pt]{dd}{\simeq}\arrow[bend right=45,shift right=26pt,swap]{dd}{\bar{\Pi}_E}\arrow[dashed]{rr}{\De_!} \arrow[swap]{d}{\bar{\Pi}}\arrow{d}{\simeq}\& \& \E\otimes _{\B}\E[-2d] \arrow[bend left=45,shift left=36pt]{dd}{\bar{\Pi}_{E\times_BE}}\arrow[bend left=45,shift left=36pt,swap]{dd}{\simeq} \arrow{d}{\overline{\Pi\otimes \Pi}} \arrow[swap]{d}{\simeq}\\
 \Hom_{\B}(\E,\B)\arrow{rr}{\De^*}\arrow[swap]{d}{(\Pi_B)_*} \& \& \Hom_{\B}(\E\otimes_{\B}\E,\B)\arrow{d}{(\Pi_B)_*}\\
 \Hom_{\Q}(\E,\Q) \arrow{rr}{\De^*} \& \& \Hom_{\Q}(\E\otimes_{\B}\E,\Q)
 \end{tikzcd}
\end{equation*}
where the dashed maps denote the Umkehr map from \eqref{UmkehrDiagram}. The upper square commutes up to homotopy by construction and therefore so does the outer square by commutativity of the lower square. This shows that $\De_!$ is a cochain level representative of the Gysin map and thus $[\mu(\De_!(1))]\in H^d(S)$ agrees with \eqref{efwM}.
\end{proof}

\subsection{The fibrewise Euler class of Leray--Hirsch fibrations}
In the case of fibrations $\pi:E\ra B$ with oriented Poincar\'e fibre $(X,\varepsilon_X)$ which are Leray--Hirsch, i.e.\,where the restriction map $H^*(E)\ra H^*(X)$ is surjective, the definition of fibre integration and the fibrewise Euler class can be simplified significantly. Surjectivity of the restriction map implies that $H^*(E)$ is a free $H^*(B)$-module, and we may denote by $1,e_1,\hdots , e_k\in H^*(E)$ a $H^*(B)$-basis of the cohomology of $E$ that restricts to a basis $1,x_1=i^*(e_1),\hdots,x_k=i^*(e_k)\in H^*(X)$ of the fibre. If $X$ is a Poincar\' e complex of formal dimension $d$, we can order the basis such that $|e_k|=d$ and all other $|e_i|$ have lower degree. Since fibre integration is a $H^*(B)$-module map. It suffices to determine $\pi_!$ on a basis and for degree reasons $\pi_!(e_i)=0 $ for $i<k$. If we set $\pi_!(e_k)=\varepsilon_X(x_k)$ this restricts to the orientation on the fibre hence determines fibre integration as
\begin{equation}\label{fibreinterationLH}
\pi_!\left(\sum_{i=0}^kb_i\cdot e_i\right)=\varepsilon_X(x_k)\cdot b_k
\end{equation}
for $b_i\in H^*(B)$. Since the fibre is Poincar\' e, the (fibrewise) intersection pairing
	\begin{equation}\label{nondegeneracy}
\langle \dash, \dash \rangle :H^*(E)\otimes_{H^*(B)}H^*(E)\xrightarrow{\,\pi_!(\dash\cup\dash)\,} H^*(B)
\end{equation}
is non-degenerate. This enables us to mimic the construction of the Euler class as the dual  of the fibrewise diagonal.
\begin{prop}[\cite{RW16}]\label{LerayHirsch}
 Let $\pi:E\ra B$ be an oriented fibration with Poincar\'e fibre which is Leray--Hirsch. Let $e_0,\hdots,e_k\in H^*(E)$ be an $H^*(B)$-module basis and denote by $e_0^{\#},\hdots,e_k^{\#}\in H^*(E)$ the dual basis under the non-degenerate pairing \eqref{nondegeneracy}. Then the fibrewise Euler class is
 \begin{equation}\label{LHEulerClass}
  e^{\emph{fw}}(\pi)=\sum_{i=0}^k(-1)^{|e_i|}e_ie_i^{\#}\in H^d(E;\Q).
 \end{equation}
\end{prop}

\begin{ex}\label{EulerClassEvenSphere}
Let $X=S^{2n}$ be an even dimensional sphere and recall the cdga model and fibre integration from the Examples \ref{EvenSphereFibrationModel} and \ref{EvenSphereFibreIntegration}. Then $1$ and $x$ are a $H^*(\Bhaut_0(S^{2n});\Q)$-basis of the cohomology of the total space that restricts to a basis of $H^*(S^{2n})$ on the fibre, i.e. the fibration is Leray-Hirsch. Note that the formula for fibre integration in \eqref{fibreinterationLH} gives the same result as our construction of $\Pi$ in Example \ref{EvenSphereFibreIntegration}. We can apply the above Proposition to find a representative of the fibrewise Euler class. The dual basis with respect to the pairing induced by $\pi_!=H(\Pi)$ is $x^{\#}=1 $ and $1^{\#}=x $ (since $\pi_!(x\cdot 1^{\#})=\pi_!(x^2)=0$), and we find that the fibrewise Euler class is represented by $e^{\text{fw}}(\pi)=2x$.
\end{ex}

\subsection{Fibrations with positively rationally elliptic fibre}

A simply connected space $X$ is called \emph{rationally elliptic} if both $\dim H(X;\Q)$ and $\dim \pi(X)\otimes \Q$ are finite dimensional vector spaces. If moreover $\chi(X)>0$ then $X$ is called \emph{positively rationally elliptic}. Algebraic models for (positively) elliptic spaces are quite rigid. For example, one can show that any positively rationally elliptic space satisfies rational Poincar\'e duality. The main result in this section is a simple, closed formula for the fibrewise Euler class of a fibration with positively rationally elliptic fibre. 

\medskip
The minimal Sullivan model of a rationally elliptic space $\Lambda=(\Lambda V,d)$ is free on a finite dimensional vector space $V$. Hence, the dg Lie algebra model $\Der^+(\Lambda)$ for $\Bhaut_0(X)$ is finite dimensional which makes the study of fibrations with rationally elliptic fibres tractable via Theorem \ref{modelcdga}.
Moreover, a famous conjecture due to Halperin states that any fibration with positively elliptic fibre (and trivial holonomy action) is Leray-Hirsch. This conjecture is known to be true for a large number of examples \cite{Me82,ST87,Th81} and by \cite{Me82} it is equivalent to $\pi_{2i-1}(\Bhaut_0(X))\otimes \Q=0$ for all $i\in \N$. Since $\pi_{2i-1}(\Bhaut_0(X))\otimes \Q=H_{2i}(\Der^+(\Lambda),[d,-])$, this condition can easily be checked in examples. 

\smallskip
We recall a few results about positively rationally elliptic space from \cite[Sect.\,32]{FHT}. A pure Sullivan algebra is a cdga $\Lambda=(\Lambda V,d)$ with $d(V^{\text{even}})=0$ and $d(V^{\text{odd}})\subset \Lambda V^{\text{even}}$, and we denote $P=V^{\text{odd}}$ and $Q=V^{\text{even}}$. Oberseve that pure Sullivan algebras $\Lambda$ are bigraded with additional lower grading given by $\Lambda_k=\Lambda Q\otimes \Lambda^k P$. By \cite[Prop.\,32.10]{FHT} a minimal Sullivan model of a positively rationally elliptic space is isomorphic to to a pure Sullivan algebra with $\dim P=\dim Q$ and so that $d|_{P}$ maps a basis of $P$ to a regular sequence of the graded polynomial ring $\Lambda Q$. In this case, the cohomology is concentrated in lower grading $0$, i.e.\,$H^*(\Lambda)=H^*_0(\Lambda)$.

This bigrading is inherited by the dg Lie algebra $\Der^+_{*,*}(\Lambda)$ where a derivation $\theta$ has bidegree $(m,n)$ if $\theta$ lowers the internal degree by $|m|$ and $\theta(\Lambda Q\otimes \Lambda^k P)\subset \Lambda Q\otimes \Lambda^{k+n} P$.  The following statement has been explained to the author by Alexander Berglund.
 \begin{lem}\label{ModelElliptic}
    Let $\Lambda$ be a pure Sullivan model for a positively rationally elliptic space that satisfies the Halperin conjecture. There exists an abelian dg Lie algebra with trivial differential $\mathfrak{a}\subset \Der_{*,-1}^+(\Lambda)$ that is quasi-isomorphic to $\Der^+(\Lambda)$.
\end{lem}
\begin{proof}
If $\Lambda$ is a pure Sullivan model for a positively rationally elliptic space, the projection map $\psi:\Lambda\ra H(\Lambda)$ is a quasi-isomorphism. It induces a quasi-isomorphism of chain complexes $\Der^+(\Lambda)\lra \Der_{\psi}^+(\Lambda,H(\Lambda))$ (analogous to \cite[Lem.\,3.5]{BM14}). Since the $H(\Lambda)=H_0(\Lambda)$, we see that $H_{*,>0}(\Der^+(\Lambda))$ is contained in the kernel and thus is trivial. The Halperin conjecture implies that $H_*(\Der^+(\Lambda))$ is concentrated in odd degrees and therefore $H_*(\Der^+(\Lambda))\cong H_{\mathrm{odd},-1}(\Der^+\Lambda)$ as claimed. Hence, any choice of representatives in $\Der_{*,-1}^{+}(\Lambda)$ of a basis of $H_*(\Der^+(\Lambda))$ spans an abelian dg Lie subalgebra with trivial differential that is quasi-isomorphic to $\Der^+(\Lambda)$.
\end{proof}
Halperin's conjecture implies that $\Bhaut_0(X)\simeq _{\Q}\prod_{i=1}^kK(\Q,2n_i)$ and so the cohomology of the total space $E$ of \eqref{UnivXFibr1connected} is a free module on $H^*(X;\Q)$ over a positively and evenly graded polynomial ring. Moreover, it follows from Lemma \ref{ModelElliptic} and Theorem \ref{modelcdga} that as a ring $H^*(E)$ is a complete intersection over $H^*(\Bhaut_0(X);\Q)$ (see proof of Theorem \ref{completeIntersection}). Here, by a \emph{complete intersection} over a commutative ring $\B$ we mean a finite $\B$-algebra $\E$ that is isomorphic to $\B[x_1,\hdots,x_n]\slash (f_1,\hdots,f_n)$ for $f_1,\hdots,f_n\in \B[x_1,\hdots,x_n]$ (see \cite{SL97}).

For a finite $\B$-algebra $\E$, we can define the trace of an endomorphism $\Hom_{\B}(\E,\E)\cong \E\otimes_{\B}\Hom_{\B}(\E,\B)$ via the evaluation $\E\otimes_{\B}\Hom_{\B}(\E,\B)\ra \B$. In particular, we can associate to any $s\in \E$ the trace of the endomorphism $s\cdot -:\E\ra \E$ and we obtain an element $\Tr_{\E\slash \B}\in\Hom_{\B}(\E,\B)$. We will use the following result about complete intersections.

\begin{prop}[\cite{SL97}]\label{PropComplInt}
 Let $\B$ be a commutative ring and $f_1,\hdots,f_n\in \B[x_1,\hdots,x_n]$ for a non-negative integer $n$. Assume that $\E=\B[x_1\hdots,x_n]\slash (f_1,\hdots,f_n)$ is a finite $\B$-algebra. Then
 \begin{itemize}
  \item[(i)] $\E$ is a projective $\B$-module\footnote{In our applications, $\B$ is a positively graded polynomial ring so that $\E$ is in fact free.};
  \item[(ii)] $\Hom_{\B}(\E,\B)$ is a free of rank $1$ as $\E$-module;
  \item[(iii)] there is a generator $\lambda$ of $\Hom_{\B}(\E,\B)$ as an $\E$-module such that $\Tr_{\E\slash\B}=\det(\partial f_i\slash\partial x_j)\cdot \lambda$.
 \end{itemize}
\end{prop}
We recognize (iii) as an analogue for complete intersections of the relation between fibre integration, the fibrewise Euler class and the Becker-Gottlieb transfer $\tau_{\pi}:\Sigma^{\infty}B_+\ra \Sigma^{\infty}E_+$ for fibrations with Poincar\'e fibres. The transfer map induces a map on cohomology $\text{trf}^*_{\pi}:H^*(E)\ra H^*(B)$ and if the fibre is a Poincar\'e complex one can show that
\begin{equation*}
 \text{trf}^*_{\pi}(x)=\pi_!(e^{\fw}(\pi)\cdot x)
\end{equation*}
 for all $x\in H^*(E)$. If we consider the universal 1-connected fibration \eqref{UnivXFibr1connected} of a positively rationally elliptic space that satisfies the Halperin conjecture, its algebraic model $\B\ra \E$ is equivalent to a complete intersection over a polynomial ring and we can identify the Becker-Gottlieb transfer and fibre integration with the trace $\mathrm{Tr}_{\E/\B}$ and $\lambda$ respectively, which therefore leads to an identification of the fibrewise Euler class.
\begin{thm}\label{completeIntersection}
Let $X$ be a simply connected, oriented Poincar\'e duality space that is positively rationally elliptic and satisfies the Halperin conjecture. Then a cdga model of \eqref{UnivXFibr1connected} is given by a complete intersection $\E=\B[x_1,\hdots,x_n]/(f_1,\hdots,f_n)$ over a polynomial ring $\B$ and the fibrewise Euler class is given by 
 \begin{equation}\label{efwCI}
  e^{\emph{fw}}(\pi)=\det\left(\frac{\partial f_i}{\partial x_j}\right)\in \E.
 \end{equation}
\end{thm}
\begin{rem}
  The first part of the theorem improves a result by Kuribayashi \cite[Thm 1.1]{Ku10} where he showed that the cohomology ring of the total space of the universal 1-connected fibration \eqref{UnivXFibr1connected} for positively rationally elliptic spaces that satisfy the Halperin conjecture is a complete intersection over the cohomology of the base. 
\end{rem}
\begin{proof}
  Let $\Lambda=(\Q[x_1,\hdots,x_n]\otimes \Lambda(y_1,\hdots,y_n),d)$ be a pure Sullivan model of $X$ with $d(y_i)=\bar{f}_i\in \Q[x_1,\hdots,x_n]$ where $\bar{f}_1,\hdots,\bar{f}_n$ is a regular sequence. The cohomology ring $H(\Lambda)$ is a Poincar\'e duality algebra and $\det (\partial \bar{f}_i\slash \partial x_j)\in H^*(X)$ is a generator in top degree \cite{Mu93,ST87}, so that the orientation is defined by $\varepsilon_X(\det (\partial \bar{f}_i\slash \partial x_j)):=\chi(X)$. Now let $\mathfrak{a}\subset \Der^+_{*,-1}(\Lambda)$ be an abelian dg Lie subalgebra quasi-isomorphic to $\Der^+(\Lambda)$ from Lemma \ref{ModelElliptic}. Then a model of the universal 1-connected fibration is given by
  \[\B:=\cC_{CE}(\mathfrak{a})\ra \cC_{CE}(\mathfrak{a};\Lambda)\cong (\B\otimes \Lambda(x_i,y_i),D)
  \]
  by Theorem \ref{modelcdga}. And because $\mathfrak{a}\subset \Der^+_{*,-1}(\Lambda)$, we see that $D(x_i)=0$ and that $D(y_i)=f_i\in \B[x_1,\hdots,x_n]$ are polynomials that satisfy $\bar{f}_i=f_i\otimes_{\B}1\in \B[x_1,\hdots,x_n]\otimes_{\B}\Q$. This implies that $\E:=\B[x_1,\hdots,x_n]/(f_1,\hdots,f_n)$ is a complete intersection over $\B$ and the projection map $\cC_{CE}(\mathfrak{a};\Lambda)\ra \E$ is a quasi-isomorphism.

 One can show that the trace $\mathrm{Tr}_{\E\slash\B}$ and the transfer $\mathrm{tr}_{\pi}^*$ agree using the general theory in \cite{DP80} or by directly checking that 
 \begin{equation}\label{lambda=pi}
 \Tr_{\E\slash \B}( x)=\pi_!(e^{\text{fw}}(\pi)\cdot x),
 \end{equation}
 as has been done in \cite[Lemma 2.3]{RW16} using that $\E$ is a finite free $\B$-algebra. Now consider $\lambda$ from Proposition \ref{PropComplInt}, then $\lambda(\det(\partial f_i\slash \partial x_j) )=\mathrm{Tr}_{\E\slash\B}(1)=\chi(X)$ by \eqref{lambda=pi} which agrees with $\pi_!(\det(\partial f_i\slash \partial x_j))=\varepsilon_X(\det(\partial \bar{f}_i\slash \partial x_j))$. Therefore, it follows from degree reasons that $\pi_!=\lambda$ and we obtain \eqref{efwCI} from the identity 
 \begin{equation*}
\det(\partial f_i\slash\partial x_j)\cdot \pi_!=\det(\partial f_i\slash\partial x_j)\cdot \lambda =\Tr_{E\slash B}=\text{trf}_{\pi}^*=e^{\text{fw}}(\pi)\cdot \pi_!
\end{equation*}
since $\Hom_{\B}(\E,\B)$ is a free $\E$-module by (ii).
\end{proof}

\section{Computations}\label{computations}

Before we come to the computation of the Euler ring, we record a few general facts. First we observe that in some cases it is sufficient to compute the Euler ring of the universal 1-connected fibration $E_0(X)\subset H^*(\Bhaut_0(X))$. 
\begin{lem}\label{finiteE}
 If $\pi_0(\haut^+(X))$ is finite, then $i:\Bhaut_0(X)\ra \Bhaut^+(X)$ induces an isomorphism $E^*(X)\cong E_0^*(X)$.
\end{lem}
\begin{proof}
 By naturality of the fibrewise Euler class, $i:\Bhaut_0(X)\ra\Bhaut^+(X)$ induces a surjection of Euler rings $E^*(X)\twoheadrightarrow E_0^*(X)$. It follows from the the spectral sequence associated to the fibre sequence 
 \[\Bhaut_0(X)\lra \Bhaut^+(X)\lra \mathrm{B}\pi_0(\haut^+(X))\]
 that $H^*(\Bhaut(X);\Q)\cong H^*(\Bhaut_0(X);\Q)^{\pi_0(\haut^+(X))}$ if $\pi_0(\haut^+(X))$ is finite. In particular, $i$ induces an injection on rational cohomology.
\end{proof}
For a positively rationally elliptic space there is a simple criterion to check if it has finitely many homotopy automorphisms.
\begin{prop}\label{finite}
	Let $X$ be a simply connected Poincar\'e duality space that is positively rationally elliptic space. If the cohomology ring $H^*(X;\Q)$ has finitely many orientation preserving automorphisms (i.e.\,algebra automorphisms that preserve the orientation $\varepsilon_X\in H^d(X;\Q)^{\vee}$), then $\pi_0(\haut^+(X))$ is finite.
\end{prop}
\begin{proof}
The statement follows if the kernel of $\mathcal{E}(X)\ra \Aut(H^*(X;\Q),\varepsilon_X)$ is finite. Rationalization induces a map $\mathcal{E}(X)\ra \mathcal{E}(X_{\Q})$ which has finite kernel \cite[Thm 10.2]{Sul77}, so that it suffices to prove that the subgroup of $\mathcal{E}(X_{\Q})$ of homotopy equivalences that induce the identity on cohomology is trivial. By \cite[10.3]{Sul77}, homotopy automorphism of $X_{\Q}$ are the same as homotopy classes of automorphisms of a minimal Sullivan model. 	
	
For a positively rationally elliptic space we can choose as a model $\Lambda=(\Lambda(x_i,y_i)_{i=1,\hdots,n},d=\sum_if_i\partial y_i)$ where $f_1,\hdots,f_n\in \Lambda(x_i)_{i=1,\hdots,n})$ is a regular sequence. Let $\phi:\Lambda\ra \Lambda$ be an automorphism so that $H(\phi)=\Id$. Then $[x_i-\phi(x_i)]=0$ and we can pick a coboundary $\xi_i\in \Lambda_{>0}$ satisfying $d\xi_i=x_i-\phi(x_i)\in (f_1,\hdots,f_k)+\Lambda_{>0}$. We want to define a homotopy $H:\Lambda\ra \Lambda\otimes \Lambda(t,dt)$ by setting $H(x_i)=\phi(x_i)+t(x_i-\phi(x_i))-\xi_idt$. Then $dH(x_i)=0=Hd(x_i)$ and $\varepsilon_1\circ H(x_i)=x_i$ and $\varepsilon_0 \circ H(x_i)=\phi(x_i)$. It remains to define $H(y_i)$, i.e.\,we have to find a coboundary for $f_i(H(x_1),\hdots,H(x_n))\in \Lambda\otimes \Lambda(t,dt)$. Observe that $\Lambda\otimes \Lambda(t,dt)$ is a positively elliptic Sullivan algebra as well since $f_1,\hdots,f_n,t$ is a regular sequence in $\Q[x_1,\hdots,x_n,t]$, and that $f_i(H(x_1),\hdots,H(x_n))-d\phi(y_i)\in (\Lambda\otimes \Lambda^+(t,dt)) \oplus (\Lambda\otimes \Lambda(t,dt))_{>0}$. Since both summands are acyclic, there is $\zeta_i\in \Lambda\otimes \Lambda(t,dt)$ so that $d\zeta_i=f_i(H(x_1),\hdots,H(x_n))-d\phi(y_i)$, and we can set $H(y_i):=\zeta_i+\phi(y_i)$. This shows that $\phi$ is homotopic to a map which is the identity on $\Lambda(x_i)$. 

Given an automorphism $\psi:\Lambda\ra\Lambda$ with $\psi(x_i)=x_i$, then $d\psi(y_i)=\psi(f_i)=f_i=d(y_i)$ and hence $y_i-\psi(y_i)$ is a cocylce in $\Lambda_{>0}$. Let $\xi_i$ be a coboundary $d\xi_i=y_i-\psi(y_i)$, then $\psi\simeq_H \Id_{\Lambda}$ via $H(x_i):=x_i$ and $H(y_i):=\psi(y_i)+t(y_i-\psi(y_i))-\xi_idt$.
\end{proof}
\begin{cor}\label{RationalProjective}
	Let $X$ be a simply connected Poincar\'e duality space with $H^*(X;\Q)\cong \Q[x]/(x^{n+1})$ for some $n$ and $|x|$ even, then $\pi_0(\haut^+(X))$ is finite.  
\end{cor}
\begin{proof}
 For any choice of orientation of $X$, the group  $\Aut(H^*(X;\Q),\varepsilon_{X})$ is trivial if $n$ is odd and $\Z/2$ if $n$ is even. Hence, $\pi_0(\haut^+(X))$ is finite by Proposition \ref{finite}
\end{proof}

Finally, in some cases the Euler ring is finitely generated so that the computation of the Euler ring amounts to computing the ideal of relations among a generatoring set, and which simplifies some computations below.
\begin{prop}\label{FiniteGeneration}
 Let $X$ be a Poincar\'e duality space so that $H^*(X)$ is concentrated in even degrees and let $n=\dim H^*(X;\Q)$. Then $E^*(X)$ is generated by $\kappa_1,\hdots,\kappa_{n-2},\kappa_{n}$.
\end{prop}
\begin{proof}
This follows directly from the tools developed in \cite{RW16}. Applying \cite[Cor.\,2.7]{RW16} for $x=e^{\fw}(\pi)$, the corresponding monic polynomial $\rho_x(z)\in H^*(\Bhaut^+(X);\Q)[z]$ of degree $n$ is given by
\[\rho_x(z)=\frac{(-1)^n}{n!}\sum_{\sigma\in \Sigma_{n+1}} \mathrm{sgn}(\sigma) \kappa_{l(\gamma_2)}\cdot \dots\cdot \kappa_{l(\gamma_{q(\sigma)})} \cdot z^{l(\gamma_1)-1},\]
where $\sigma=\gamma_1\cdot \dots\cdot \gamma_{q(\sigma)}$ is the cycle decomposition of $\sigma$, $l(\gamma_i)$ denotes the length of $\gamma_i$ and $1$ is contained in the support of $\gamma_1$. Then \cite[Cor.\,2.7]{RW16} implies that $\rho_x(e^{fw}(\pi))=0\in H^{*}(E)$ and by fibre integrating $0=e^{\fw}(\pi)^i\cdot \rho(e^{\fw}(\pi))$, one can decompose $\kappa_{n+i}$ in terms of $\kappa$-classes of lower degree for $i\geq 0$; except when $i=1$ as $\rho_x(z)$ has a constant term$-\kappa_n/n$ which cancels the leading term $\pi_!(e^{\fw}(\pi)^{n+1})$ in $\pi_!(e^{\fw}(\pi)\rho_x(e^{\fw}(\pi)))$.
\end{proof}

\subsection{The Euler ring of even spheres}

\begin{prop}
 The Euler ring of an even dimensional sphere is $E^*(S^{2n})\cong \Q[\kappa_2]$ where $\kappa_{2}^k=2^{k-1}\kappa_{2k}$ and all odd $\kappa_{2i+1}$ vanish.
\end{prop}
\begin{proof}
 We have seen in Example \ref{EulerClassEvenSphere} that $e^{\text{fw}}(\pi)=2x\in \Lambda(x,y,z_{4n})$ and in Example \ref{EvenSphereFibreIntegration} that fibre integration is given by $\Pi(x^{2k})=0$ and $\Pi(x^{2k+1})=(-1)^kz_{4n}^k$. Thus $\kappa_{2k}=2^{1-k}(-2^3z_{4n})^k=2^{1-k}\kappa_2^k$. Since $\pi_0(\haut^+(S^{2n}))$ is trivial, the 1-connected universal fibration is the universal fibration and the result follows.
\end{proof}

\subsection{The Euler ring of complex projective space}
A minimal model of $\C P^n$ is 
\[P_n:=(\Lambda(x,y),|x|=2,|y|=2n+1,d=x^{n+1}\partial\slash\partial y)\]
with orientation $\varepsilon_{\C P^n}(x^n)=1$ induced by integral Poincar\'e duality, and we apply Theorem \ref{modelcdga} to compute a model for the minimal 1-connected $\C P^n$-fibration. 
\begin{prop}\label{ModelProjective}
 A cdga model of the universal 1-connected $\C P^n$-fibration is given by
 \begin{align}\label{cdgaModelProjective}
  B_n:=\big(\Q[x_2,\hdots,x_{n+1}],|x_i|=2i,d=0\big)\overset{\pi}{\lra} E_n:=\big(B_n[x]\slash (x^{n+1}+\mbox{\small $\sum_{i=2}^{n+1}$} x_i\cdot x^{n+1-i}),|x|=2,d=0\big),
 \end{align}
 and the fibrewise Euler class in $E_n$ is represented by 
 \begin{equation}\label{EulerClassProjective}
e^{\emph{fw}}(\pi)=(n+1)\cdot x^n+\sum_{i=2}^{n}(n+1-i)\cdot x_i\cdot x^{n-i} \in E_n.
 \end{equation}
\end{prop}
\begin{proof}
Note that $\Der^+(P_n)$ has a (vector space) basis given by $\theta_i:=x^{n+1-i}\partial\slash \partial y$ for $i=1,\hdots, n+1$ of degree $2i-1$ and $\eta:=\partial\slash\partial x$ of degree $2$. The only non-trivial differential on the derivation Lie algebra is given by $[d,\eta]=-(n+1)\theta_1$. Since $\C P^n$ satisfies the Halperin conjecture, it follows from Lemma \ref{ModelElliptic} that there exists a quasi-isomorphic abelian dg Lie subalgebra $\mathfrak{a}_n\subset \Der^+(P_n)$ with trivial differential, which in this case is easy to identify as $\mathfrak{a}_n:=\Q\{\theta_2,\hdots,\theta_{n+1}\}$. The statement follows directly from Theorem \ref{completeIntersection} for this choice of $\mathfrak{a}_n$.
\end{proof}

\begin{thm}\label{EulerRingProjective}
	The Euler ring of complex projective space is $E^*(\C P^n)\cong \Q[\kappa_1,\hdots,\kappa_{n-1},\kappa_{n+1}]$. 
\end{thm}
\begin{proof}
	By Example \ref{RationalProjective}, $\pi_0(\haut^+(\C P^n)$ is finite (in fact one can easily see that $\pi_0(\haut(\C P^n))\cong\Z/2$), and so $E^*(\C P^n)\cong E^*_0(\C P^n)$ by Lemma \ref{finiteE}. The Euler ring is generated by $\kappa_1,\hdots,\kappa_{n-1},\kappa_{n+1}$ by Proposition \ref{FiniteGeneration}, so it remains to show that they are algebraically independent which follows if $\det (\partial \kappa_i\slash \partial x_j)$ is non-zero. 
	
	It turns out that the polynomials representing the $\kappa_i$ are quite complicated so that it is difficult to give a closed formula for the determinant of the Jacobian. We will resolve this issue by focussing on the terms containing $x_{n+1}$ because it is the variable of the highest degree and it is not contained in  $e^{\text{fw}}(\pi)$ so that it only arises through fibre integrating $x^{k}$ for $k>n$. It will be sufficient to consider elements modulo decomposables, i.e. for $x,y\in B_n$ then $x\sim y$ if $x-y\in (B_n^+)^2$. We will start with the following observation about fibre integration.
	 
	\begin{itemize}
		\item[\textbf{1.}] If $k=2,\hdots,n+1$ then $\pi_!(x^{n+k})\sim- x_k\in B_n$.
	\end{itemize}
	\noindent
	 \emph{Proof. }Rewriting $x^{n+2}$ in terms of the module basis $\{1,x,\hdots,x^n\}$, one can see that ${\pi_!(x^{n+2})=x_2}$. Then $\pi_!(x^{n+k})=\pi_!(x^{n+1}\cdot x^{k-1})=-\sum_{i=2}^{n+1}x_i\cdot \pi_!(x^{n+k-i})$ and by induction over $k$, the only indecomposable contribution is for $i=k$.\hfill $\square$
	 
	\begin{itemize}
		\item[\textbf{2.}] For $i=1,\hdots,n-1$ the highest power of $x_{n+1}$ in $\kappa_i(x_2,\hdots,x_{n+1})$ is $i-1$ and the coefficient $c_i\in \Q[x_2,\hdots,x_n]$ of $x_{n+1}^{i-1}$ satisfies $c_i\sim (-1)^{i}i(n+1)^i(n-i)\cdot x_{n+1-i}$.
	\end{itemize} 
	\noindent
	\emph{Proof.} It follows from degree considerations that the highest power of $x_{n+1}$ is $i-1$ and $c_i=A\cdot x_{n+1-i}+\text{decomposables}$. It remains to determine the coefficient $A$. When expanding $e^{\text{fw}}(\pi)$ using \eqref{EulerClassProjective}, the only relevant contributions are
	\begin{align*}
	&(n+1)^{i+1}x^{n(i+1)}-(i+1)(n+1-(n+1-i))x_{n+1-i}x^{n-(n+1-i)}\cdot (n+1)^i x^{ni}\\
	=&(n+1)^{i+1}(x^{n+1})^{i-1}\cdot x^{2n-i+1}-i(i+1)(n+1)^ix_{n+1-i}\cdot(x^{n+1})^{i-1}\cdot x^n
	\end{align*}
	Now rewrite $x^{n+1}=\sum_{i=2}^{n+1}x_ix^{n+1-i}$ and collect all terms containing $x_{n+1}^{i-1}$ and $x_{n+1}^{i-2}x_{n+1-i}$ (we can ignore the rest because it cannot contribute to A) to get
	\begin{align*}
	(n+1)^{i+1}(x_{n+1}^{i-1}+(i-1)x_{n+1}^{i-2}x_{n+1-i}\cdot x^i)\cdot x^{2n-i+1}-i(i+1)(n+1)^ix_{n+1-i}\cdot x_{n+1}^{i-1}\cdot x^n
	\end{align*}
	The statement follows by fibre integrating and discarding decomposables as in \textbf{1} above. \hfill $\square$
		\begin{itemize}
		\item[\textbf{3.}] The highest contribution of $x_{n+1}$ in $\kappa_{n+1}$ is $(-1)^{n}(n+1)^{n+2}\cdot x_{n+1}^n$.
	\end{itemize} 
	\noindent
	\emph{Proof.} The expression for $e^{\text{fw}}(\pi)^{n+2}$ contains the summand $(n+1)^{n+2}\cdot x^{n(n+2)}=(n+1)^{n+2}\cdot(x^{n+1})^n\cdot x^n$. This is the only summand that fibre integrates to a multiple of $x_{n+1}^n$, i.e. $\kappa_{n+1}=(-1)^{n}(n+1)^{n+2}\cdot x_{n+1}^n+\hdots$ where we can ignore all other terms. \hfill $\square$
	
	\medskip
	\noindent
	We can now analyze $\det(\partial \kappa_i\slash\partial x_j)$ which contains the summand
	\begin{equation*}
	\frac{\partial \kappa_1}{\partial x_n}\cdot \frac{\partial \kappa_2}{\partial x_{n-1}}\cdot\hdots\cdot\frac{\partial \kappa_{n-1}}{\partial x_2}\cdot \frac{\partial \kappa_{n+1}}{\partial x_{n+1}}.
	\end{equation*}
	It follows from \textbf{2} and \textbf{3} that the above expression contains $C\cdot x_{n+1}^N$, where $C$ is a non-zero constant and $N=\half n(n-1)$. This is the only possible way to get a monomial in $\text{det}\,\big( \frac{\partial \kappa_i}{\partial x_j}\big)$ that contains only $x_{n+1}$. Hence, the determinant does not vanish and the generating set $\kappa_1,\hdots,\kappa_{n-1},\kappa_{n+1}$ is algebraically independent.
\end{proof}

\begin{rem}
 Theorem \ref{EulerRingProjective} has been studied in the smooth case for $n=2$ in \cite{RW16} by studying the natural smooth $2$-torus action on $\C P^2$ and that implies our result in this case as well. This has been extended by Dexter Chua to $n\leq 4$, but for large $n$ the algebra becomes intractable.
\end{rem}

\subsection{The Euler ring of products of odd spheres}
The main result of this section is the computation of $E_0^*(X)$ for a simply connected Poincar\'e duality space $X$ that is rationally equivalent to a product of odd spheres $\prod_{i=1}^{2n}S^{2k_i+1}$ for some $n, k_i>0$. 
\begin{thm}\label{EulerRingOddSpheres}
 Let $X$ be a simply connected Poincar\'e duality space that is rationally equivalent to a product of odd spheres $\prod_{i=1}^{2n}S^{2k_i+1}$ for some $n, k_i>0$. Then the fibrewise Euler class of the universal 1-connected $X$-fibration $E\ra \Bhaut_0(X)$ is trivial and hence $E_0^*(X)\cong \Q$.
\end{thm}
\begin{proof}
 The minimal model of $X$ is given by an exterior algebra $A_X=(\Lambda(x_i)_{1\leq i\leq 2n},d=0)$ and by Theorem \ref{modelcdga} the model of the universal 1-connected fibration is given by 
 \begin{equation}
  B_X:=\cC^*_{CE}(\Der^+(A_X);\Q)\lra E_X:=\cC^*_{CE}(\Der^+(A_X);A_X)\cong (B_X\otimes A_X,D).
 \end{equation}
Observe that $\cC^*_{CE}(\Der^+(A_X);A_X)$ is a finitely generated $B_X$-module as the minimal Sullivan model is finite dimensional. Let $\varepsilon_X:A_X\ra \Q$ denote a rational orientation, then $B_X\otimes \varepsilon_X:(B_X\otimes A_X,D)\ra B_X$ is the only module homomorphism that restricts to $\varepsilon_X$ on the fibre and thus is a representative of fibre integration by Proposition \ref{ExtModuleMaps}.

Moreover, $\Hom_{B_X}(E_X,B_X)$ is a minimal semifree module and thus the qua\-si-iso\-mor\-phism $\bar{\Pi}:E_X[-d]\ra \Hom_{B_X}(E_X,B_X)$ from Proposition \ref{fwPD} is in fact an isomorphism by uniqueness of minimal free resolutions \cite[Ex.8,Ch.6]{FHT} (similarly for $(\overline{\Pi\otimes \Pi}$). Hence, the algebraic Umkehr map is given by
  \begin{equation*}
   \De_!:E_X[-|F|]\underset{\cong}{\overset{\bar{\Pi}}{\lra}} \Hom_{B_X}(E_X,B_X)\overset{\De^*}{\lra}\Hom_{B_X}(E_X\otimes_{B_X}E_X,B_X)\underset{\cong}{\overset{(\overline{\Pi\otimes \Pi})^{-1}}{\lra}}E_X\otimes_{B_X}E_X[-2|F|].
  \end{equation*}
  The composition of $\Bar{\Pi}$ with the vector space isomorphism $\Hom_{B_X}(E_X,B_X)\cong (A_X)^{\vee} \otimes B_X$ is given by $\bar{\varepsilon}_{X}\otimes \Id_{B_X}$ where $\bar{\varepsilon}_X:A_X\ra (A_X)^{\vee}$ is the adjoint of $\varepsilon_X:A_X\otimes A_X\ra \Q$. The same statement holds for $\overline{\Pi\otimes\Pi}$ with the appropriate choice of orientation on $X\times X$ given by $\varepsilon_{X\times X}:=\varepsilon_X\otimes \varepsilon_X:(A_X\otimes A_X)^{\otimes 2}\ra \Q$. Note that $\De^*\bar{\Pi}(1)$ is contained in $(A_X\otimes A_X)^{\vee}\otimes 1$ so that $(\overline{\Pi\otimes
 \Pi})^{-1}\De^*\bar{\Pi}(1)$ is in $A_X\otimes A_X\otimes 1\subset E_X\otimes _{B_X}E_X$. A direct computation shows that
  \begin{equation*}
   \De_!(1)=\sum_{S_1\sqcup S_2=F} \pm x_{S_1}\otimes x_{S_2}\in E_X\otimes_{B_X}E_X
  \end{equation*}
 for some signs that can be worked out. Hence, the fibrewise Euler class is $e^{\text{fw}}(\pi)=\De^*\circ \De_!(1)=\sum_{S_1\sqcup S_2=F} \pm x_{F}$ and since $\Pi(e^{fw}(\pi))=\chi(X)=0$, the summands must cancel.
\end{proof}

One can easily see that the group of homotopy self-equivalences $\pi_0(\haut^+(X))$ is not finite in most cases, and so we cannot infer that the full Euler ring is trivial in general except in the two cases below. 
\begin{thm}\label{EulerRingOddSpheresText}
	Let $X$ be either rationally equivalent to  $(S^{2k+1})^{\times n}$ or a finite CW complex rationally equivalent to $S^{2k+1}\times S^{2l+1} $ for $1<k<l$ and $n$ even. Then $E^*(X)=\Q$.
\end{thm}

\begin{proof}
    We start with the proof of the second case. By \cite[Thm.10.3]{Sul77} the group $\pi_0(\haut^+(X))$ is commensurable with an arithmetic subgroup of the homotopy classes automorphisms of $A_X$. If $X\simeq_{\Q}S^{2k+1}\times S^{2l+1}$ then the group of automorphisms of $A_X$ modulo homotopy is $\Q^{\times}\times\Q^{\times}$ and the arithmetic subgroups of this linear algebraic group are finite and hence by commensurability so is $\pi_0(\haut^+(X))$. Hence, $E^*(X)\cong E_0^*(X)=\Q$ by Lemma \ref{finiteE} and Theorem \ref{EulerRingOddSpheres}. 
	
	We need to introduce some notation in the first case. Let $\pi:E\ra \Bhaut^+(X)$ denote the universal orientated $X$-fibration and $\pi_0:E_0\ra \Bhaut_0(X)$ the universal 1-connected $X$-fibration. The cdga model of the total space $E_0$ is a free algebra on generators $x_1,\hdots,x_n,y^1,\hdots,y^n$ with differential $D(x_i)=y^i$ and therefore $H(E_0)=\Q$. It is well known that $E$ is homotopy equivalent to the classifying space of pointed homotopy automorphisms $\Bhaut_*^+(X)$ and there is a fibration sequence $E_0\ra E\xrightarrow{H}\mathrm{B} \pi_0(\haut^+_*(X))$ which induces an isomorphism $H^*(E)\cong H^*(\mathrm{B} \pi_0(\haut^+_*(X)))$ since the rational cohomology of $E_0$ is trivial. In particular, $e^{\text{fw}}(\pi)=H^*e$ for some $e\in H^{n\cdot (2k+1)}(\mathrm{B}\pi_0(\haut^+_*(X)))$. It follows from the commutative diagram
	\begin{align*}
	 \begin{tikzcd}[ampersand replacement=\&]
	 E \arrow{d}{\pi} \arrow{r}{H}\arrow[swap]{r}{\simeq_{\Q}} \& \mathrm{B}\pi_0(\haut_*^+(X))\arrow[equals]{d}\\
	 \Bhaut^+(X) \arrow{r}{h} \& \mathrm{B}\pi_0(\haut^+(X)),
	 \end{tikzcd}
	\end{align*}
	that the fibrewise Euler class $e^{\text{fw}}(\pi)=H^*e=\pi^*h^*e$ is pulled back from the base, and therefore that the fibre integrals $\pi_!(\pi^*(h^*e)^k)=(h^*e)^k\cdot \pi_!(1)=0$ all vanish.	
\end{proof}
\begin{rem}
 There are more cases when the group of homotopy self-equivalences $\mathcal{E}(X)$ is finite for a space $X$ that is rationally equivalent to a product of odd dimensional spheres, and the second case in Theorem \ref{EulerRingOddSpheresText} is merely one of the simplest to establish. But we expect that the Euler ring vanishes in general even if $\mathcal{E}(X)$ is not finite using recent results in \cite{BZ22}, where Berglund and Zeman have given a rational description of $\Bhaut(X)$ via a  fibre sequence \[\Bhaut_u(X)\lra \Bhaut(X)\lra \mathrm{B}\Gamma(X),\] where $\Gamma(X)$ is a certain arithmetic group and $\Bhaut_u(X)$ is the classifying space of normal unipotent $X$-fibrations. They provide $\Gamma(X)$-equivariant models for $\Bhaut_u(X)$ so that one can obtain an algebraic model for $\Bhaut(X)$. We expect that one can extend the results in this paper using their results to study Euler rings in more generality, and in particular that the Euler ring $E^*(X)$ vanishes for $X$ a product of odd spheres in general.
\end{rem}

\bibliographystyle{alpha}
\bibliography{../../../../Bibliography/central-bib}

\begin{thebibliography}{HLLRW21}

\bibitem[{Ber}20]{Ber17}
A.~{Berglund}.
\newblock {Rational models for automorphisms of fiber bundles}.
\newblock {\em Doc. Math.}, (25):239--265, 2020.

\bibitem[Ber22]{Ber20}
A.~Berglund.
\newblock Characteristic classes for families of bundles.
\newblock {\em Selecta Math. (N.S.)}, 28(3):Paper No. 51, 56, 2022.

\bibitem[BH58]{BH58}
A.~Borel and F.~Hirzebruch.
\newblock Characteristic classes and homogeneous spaces. {I}.
\newblock {\em Amer. J. Math.}, 80:458--538, 1958.

\bibitem[BM20]{BM14}
A.~{Berglund} and I.~{Madsen}.
\newblock {Rational homotopy theory of automorphisms of manifolds}.
\newblock {\em Acta Math.}, 224(1):67--185, 2020.

\bibitem[Boa99]{Bo98}
J.~M. Boardman.
\newblock Conditionally convergent spectral sequences.
\newblock In {\em Homotopy invariant algebraic structures ({B}altimore, {MD},
  1998)}, volume 239 of {\em Contemp. Math.}, pages 49--84. Amer. Math. Soc.,
  Providence, RI, 1999.

\bibitem[{Bra}20]{BM20}
V.~{Braunack-Mayer}.
\newblock {Strict algebraic models for rational parametrised spectra II}.
\newblock {\em arXiv e-prints}, page arXiv:2011.06307, November 2020.

\bibitem[{Bra}21]{BM21}
V.~{Braunack-Mayer}.
\newblock Strict algebraic models for rational parametrised spectra, {I}.
\newblock {\em Algebr. Geom. Topol.}, 21(2):917--1019, 2021.

\bibitem[BZ22]{BZ22}
Alexander Berglund and Tomáš Zeman.
\newblock Algebraic models for classifying spaces of fibrations, 2022.

\bibitem[DP80]{DP80}
A.~Dold and D.~Puppe.
\newblock Duality, trace, and transfer.
\newblock In {\em Proceedings of the {I}nternational {C}onference on
  {G}eometric {T}opology ({W}arsaw, 1978)}, pages 81--102. PWN, Warsaw, 1980.

\bibitem[Dre67]{Dr67}
A.~Dress.
\newblock Zur {S}pectralsequenz von {F}aserungen.
\newblock {\em Invent. Math.}, 3:172--178, 1967.

\bibitem[dSL97]{SL97}
B.~de~Smit and H.~W. Lenstra, Jr.
\newblock Finite complete intersection algebras and the completeness radical.
\newblock {\em J. Algebra}, 196(2):520--531, 1997.

\bibitem[FHT01]{FHT}
Y.~F\'elix, S.~Halperin, and J.-C. Thomas.
\newblock {\em Rational homotopy theory}, volume 205 of {\em Graduate Texts in
  Mathematics}.
\newblock Springer-Verlag, New York, 2001.

\bibitem[FMT10]{FMT}
Y.~F\'elix, A.~Murillo, and D.~Tanr\'e.
\newblock Fibrewise stable rational homotopy.
\newblock {\em J. Topol.}, 3(4):743--758, 2010.

\bibitem[FT09]{FT09}
Y.~F\'{e}lix and J.-C. Thomas.
\newblock String topology on {G}orenstein spaces.
\newblock {\em Math. Ann.}, 345(2):417--452, 2009.

\bibitem[GGRW17]{GGRW17}
S.~Galatius, I.~Grigoriev, and O.~Randal-Williams.
\newblock Tautological rings for high-dimensional manifolds.
\newblock {\em Compos. Math.}, 153(4):851--866, 2017.

\bibitem[Got79]{Go79}
D.~H. Gottlieb.
\newblock Poincar\'e duality and fibrations.
\newblock {\em Proc. Amer. Math. Soc.}, 76(1):148--150, 1979.

\bibitem[Gri79]{Gr79}
P.-P. Grivel.
\newblock Formes diff\'{e}rentielles et suites spectrales.
\newblock {\em Ann. Inst. Fourier (Grenoble)}, 29(3):ix, 17--37, 1979.

\bibitem[Gri17]{Gri17}
I.~Grigoriev.
\newblock Relations among characteristic classes of manifold bundles.
\newblock {\em Geom. Topol.}, 21(4):2015--2048, 2017.

\bibitem[Hal83]{Ha83}
S.~Halperin.
\newblock Lectures on minimal models.
\newblock {\em M\'{e}m. Soc. Math. France (N.S.)}, (9-10):261, 1983.

\bibitem[HLLRW21]{HLLRW17}
F.~Hebestreit, M.~Land, W.~L\"{u}ck, and O.~Randal-Williams.
\newblock A vanishing theorem for tautological classes of aspherical manifolds.
\newblock {\em Geom. Topol.}, 25(1):47--110, 2021.

\bibitem[Kur10]{Ku10}
K.~Kuribayashi.
\newblock On the rational cohomology of the total space of the universal
  fibration with an elliptic fibre.
\newblock In {\em Homotopy theory of function spaces and related topics},
  volume 519 of {\em Contemp. Math.}, pages 165--179. Amer. Math. Soc.,
  Providence, RI, 2010.

\bibitem[Laz14]{La14}
A.~Lazarev.
\newblock Models for classifying spaces and derived deformation theory.
\newblock {\em Proc. Lond. Math. Soc. (3)}, 109(1):40--64, 2014.

\bibitem[Mei82]{Me82}
W.~Meier.
\newblock Rational universal fibrations and flag manifolds.
\newblock {\em Math. Ann.}, 258(3):329--340, 1981/82.

\bibitem[Mum83]{Mum83}
D.~Mumford.
\newblock Towards an enumerative geometry of the moduli space of curves.
\newblock In {\em Arithmetic and geometry, {V}ol. {II}}, volume~36 of {\em
  Progr. Math.}, pages 271--328. Birkh\"{a}user Boston, Boston, MA, 1983.

\bibitem[Mur93]{Mu93}
A.~Murillo.
\newblock The top cohomology class of certain spaces.
\newblock {\em J. Pure Appl. Algebra}, 84(2):209--214, 1993.

\bibitem[Pri20]{Pri20}
N.~Prigge.
\newblock {\em On tautological classes of fibre bundles and self-embedding
  calculus}.
\newblock 2020.
\newblock Thesis (Ph.D.)-- University of Cambridge.

\bibitem[{Pri}23]{Pri23}
N.~{Prigge}.
\newblock {Tautological rings of fake quaternionic spaces}.
\newblock {\em arXiv e-prints}, page arXiv:2304.13825, April 2023.

\bibitem[RW18]{RW16}
O.~Randal-Williams.
\newblock Some phenomena in tautological rings of manifolds.
\newblock {\em Selecta Math. (N.S.)}, 24(4):3835--3873, 2018.

\bibitem[ST87]{ST87}
H.~Shiga and M.~Tezuka.
\newblock Rational fibrations, homogeneous spaces with positive {E}uler
  characteristics and {J}acobians.
\newblock {\em Ann. Inst. Fourier (Grenoble)}, 37(1):81--106, 1987.

\bibitem[Sta63]{St63}
J.~Stasheff.
\newblock A classification theorem for fibre spaces.
\newblock {\em Topology}, 2:239--246, 1963.

\bibitem[{Sto}22]{Sto22}
R.~{Stoll}.
\newblock {The stable cohomology of self-equivalences of connected sums of
  products of spheres}.
\newblock page arXiv:2203.15650, March 2022.

\bibitem[Sul77]{Sul77}
D.~Sullivan.
\newblock Infinitesimal computations in topology.
\newblock {\em Inst. Hautes \'Etudes Sci. Publ. Math.}, (47):269--331 (1978),
  1977.

\bibitem[Tan83]{Ta83}
D.~Tanr\'e.
\newblock {\em Homotopie rationnelle: mod\`eles de {C}hen, {Q}uillen,
  {S}ullivan}, volume 1025 of {\em Lecture Notes in Mathematics}.
\newblock Springer-Verlag, Berlin, 1983.

\bibitem[Tho81]{Th81}
J.-C. Thomas.
\newblock Rational homotopy of {S}erre fibrations.
\newblock {\em Ann. Inst. Fourier (Grenoble)}, 31(3):v, 71--90, 1981.

\bibitem[Wal67]{Wa67}
C.~T.~C. Wall.
\newblock Poincar\'{e} complexes. {I}.
\newblock {\em Ann. of Math. (2)}, 86:213--245, 1967.

\end{thebibliography}

\end{document}